\documentclass{amsart}
 \usepackage{amssymb,stmaryrd}
 \usepackage{tikz, tikz-cd, pgfplots, amsfonts}
 \usepackage[hidelinks]{hyperref}

\newcommand{\textoverline}[1]{$\overline{\mbox{#1}}$}

 \newtheorem{theorem}{{\rm T\sc heorem}}[section]
 \newtheorem{lemma}[theorem]{{\rm L\sc emma}}
 
 \newtheorem{proposition}[theorem]{{\rm P\sc roposition}}

 \newtheorem{definition}[theorem]{{\rm D\sc efinition}}

 \theoremstyle{definition}
 \newtheorem{remark}[theorem]{{\rm R\sc emark}}
  
  \newtheorem*{acks}{Acknowledgments}

\DeclareMathOperator{\aut}{Aut}
\DeclareMathOperator{\pic}{Pic}

\DeclareMathOperator{\ns}{NS}
\DeclareMathOperator{\numer}{Num}

\begin{document}
 \title {Singular rational curves on Enriques surfaces}
 \author[Simone Pesatori]{Simone Pesatori}
 \address{Università degli Studi Roma Tre\\
         Largo San Leonardo Murialdo, 1\\
         Rome\\
         Italy}
         %\thanks{K.R. partially supported by RCN project no 239015 ``Special Geometries''.}

 \begin{abstract} We show that for every $k\in\mathbb{Z}_+$, with $k\equiv_4 1$, the very general Enriques surface admits rational curves of arithmetic genus $k$ with $\phi$-invariant equal to 2.
\end{abstract}
 \maketitle
 \pagestyle{myheadings}\markboth{\textsc{ Simone Pesatori }}{\textsc{Singular rational curves on Enriques surfaces}}
\section{Introduction}

This work addresses the topic of singular rational curves on Enriques surfaces. It is well-known that the only curves with negative self-intersection that an Enriques surface can admit are the smooth $(-2)$-curves, which turn out to be rational. The Enriques surfaces admitting such curves are called \textit{nodal} (or \textit{special}) and they form a divisor in the ten-dimensional moduli space of Enriques surfaces, whose geometry has been studied by Dolgachev and Kond\textoverline{o} in \cite{DK}, where the authors prove its rationality. It is classical that every Enriques surface is equipped with some genus 1 fibrations over a base isomorphic to $\mathbb{P}^1$, and that their singular fibers are generically genus 1 curves with a node, whence rational. Regarding higher genera, Hulek and Sch\"utt in \cite{HS} show that for every nonnegative integer $m\in\mathbb{Z}_+$ there exists an irreducible nine-dimensional family $\mathcal{F}_m$ in the moduli space of K3 surfaces parametrizing K3's covering an Enriques surface admitting a rational curve of arithmetic genus $m$, called Enriques surfaces of base change type (or \textit{$m$-special} Enriques surfaces). More precisely, by using lattice theory they prove that every Enriques surface covered by a K3 surface in $\mathcal{F}_m$ has an elliptic pencil with a rational bisection of arithmetic genus $m$ that split in the K3 cover in two smooth $(-2)$-curves. The author in \cite{Pesa} gives a lattice-free purely geometrical proof of the existence of the Enriques surfaces of base change type, allowing to show that for the general Enriques surface covered by a K3 surface in $\mathcal{F}_m$ the mentioned bisections are nodal. Moreover, in the same work the author shows that for every $m$ and for every $k\in\mathbb{Z}_+$, the very general Enriques surface $Y_m$ covered by a K3 surface in $\mathcal{F}_m$ admits nodal rational curves of arithmetic genus $(4k^2-4k+1)m+4k^2-4k$ and that for fixed $k$, the set of these bisections injects as a rank 8 subgroup in $\aut(Y_m)$. Even regarding higher genera, a result of Baltes (\cite[Theorem 1.1]{Bal}) states that every K3 surfaces with a genus 1 pencil admits rational curves of arbitrary large arithmetic genus: since any Enriques surface is covered by such a K3 surface, a consequence of the result of Baltes is that any Enriques surface also contains a rational curve of arbitrary large arithmetic genus. In this context, Galati and Knutsen in \cite{GK} state a necessary condition for the existence of a rational curve on the very general Enriques surface: they prove that every rational curve on the very general Enriques surface $Y$ is 2-divisible in $\numer(Y)$. The authors wonder what are the conditions on a line bundle on $Y$ to admit a rational member.\par
This work aims to refine the result of Baltes and to partially answer the question posed by Galati and Knutsen. In particular, once noticed that the possible arithmetic genera of a rational curve on the very general Enriques surface $Y$ are congruent to 1 modulo 4, we prove the following Theorem, stating that every possible arithmetic genus is "hit" by some rational curves on $Y$.

\begin{theorem}\label{main}
    For every $k\in\mathbb{Z}_+$ with $k\equiv_4 1$, the very general Enriques surface $Y$ admits rational curves of arithmetic genus $k$ and $\phi=2$.
\end{theorem}
Here $\phi$ is an invariant for a nef divisor $H$ on an Enriques surface, introduced by Cossec in \cite{Co}: $2\phi$ is the minimum of the intersection between $H$ and a genus 1 pencil.
\par
A consequence of the result given by Galati and Knutsen in \cite{GK} is that a rational curve on the very general Enriques surface cannot have $\phi=1$. On the other hand, all the curves found by Hulek and Sch\"utt in \cite{HS} are bisections for an elliptic pencil, whence they have $\phi=1$. We will see that two of these bisections deform together to give rise to an irreducible rational curve with $\phi=2$ on the very general Enriques surface. We will prove the phenomenon by looking at the deformations of the K3's that cover the Enriques surfaces of base change type and the very general Enriques surfaces. In particular, we exploit the "regeneration" results given by Chen, Gounelas and Liedtke in \cite{CGL} about curves on K3 surfaces. We show that the union of two suitable rational curves on $X_m\in\mathcal{F}_m$ deform to an irreducible rational curve on the very general K3 surface with an Enriques involution. Actually, our proof holds for every K3 surface that is a double covering of a smooth quadric (called hyperelliptic K3 surface). \par
In Section 2 we review the basic notions about Enriques surfaces and K3 surfaces, we describe the main features of some particular divisors on an Enriques surface and we analyze the model of a K3 surface with an Enriques involution as double cover of a smooth quadric in $\mathbb{P}^3$, called Horikawa model and presented by Barth, Hulek, Peters and Van de Ven in \cite{BHPV}. In particular, we explore the geometry of the Enriques surfaces of base change type by reading them through the mentioned tools. In Section 3 we give the notion of logarithmic Severi variety of curves on surfaces, that parametrize curves in a fixed linear system with given tangency conditions to a fixed curve. Their properties of nonemptiness, irreducibility and the computation of their dimension have been deeply analized by Dedieu in \cite{De}, and his results are based on the works of Caporaso and Harris (see for example \cite{CH}). Once pointed out the conditions that a curve on the very general Enriques surface has to satisfy to be rational, we exploit the aforementioned regeneration results due to Chen, Gounelas and Liedke to prove Theorem \ref{main}.

 \begin{acks}This work is the continuation of a part of my PhD Thesis and I want to thank my advisor Andreas Leopold Knutsen for the fruitful conversations which inspired this paper. I would also like to thank Concettina Galati for her suggestions. \end{acks}

\section{Generalities on Enriques surfaces \label{sec1}}

We recall the basics about K3 surfaces and Enriques surfaces. As general references the reader might consult \cite{BHPV} or \cite{CD}. 

\begin{definition}
    A smooth projective surface $X$ is called \textit{K3 surface}
if $X$ is (algebraically) simply connected with trivial canonical bundle $\omega_X\cong\mathcal{O}_X$.\\
An \textit{Enriques surface} $Y$ is a quotient of a K3 surface $X$ by a fixed point free involution $\tau$. Such an involution is also called \textit{Enriques involution}.
\end{definition}
For any Enriques surface $Y$ there is two-torsion in $\ns(Y)$ represented by the canonical divisor $K_Y$. The quotient $\ns(Y)_f$ of $\ns(Y)$ by its torsion subgroup is an even unimodular lattice, which is isomorphic to the so-called \textit{Enriques lattice}: \begin{center}
    $\numer(Y)=\ns(Y)_f\cong U\oplus E_8(-1)$.
\end{center}
The Picard rank of any Enriques surface is equal to 10.

\begin{definition}\label{verygenenr}
    We say that an Enriques surface $Y$ is \textit{Picard very general} if its universal covering $X$ is such that \begin{center}
        $\ns(X)\cong U(2)+E_8(-2)$.
    \end{center}
\end{definition}
\begin{remark}
    An Enriques surface $Y$ is Picard very general if and only if the Picard rank of its universal covering $X$ is equal to $10$.
\end{remark}
K3 and Enriques surfaces are the only surfaces that may admit more than one genus 1 pencil.
It is well-known that every genus 1 pencil on an Enriques surface has exactly two fibers of multiplicity two, called \textit{half-fibers}. The canonical divisor of $Y$ can be represented as the difference of the supports of the half-fibers of a genus 1 pencil: if $F$ is a genus 1 pencil of $Y$ and \begin{center}
    $2E_1\sim F$ and $2E_2\sim F$, then $K_Y\sim E_1-E_2$.
\end{center}

Ohashi in \cite{Oa} has classified K3 surfaces with Picard number $\rho=11$ and an Enriques involution. The Enriques quotient of such K3 surfaces are not Picard very general. Hulek and Sch\"utt in \cite{HS} performed a deep analysis of one type of such K3 surfaces, that they call Enriques surfaces of base change type. The following Theorem
 collects their results that we shall need in this work.
 \begin{theorem}[Hulek, Sch\"utt]\label{hs}
    For every nonnegative integer $m\in\mathbb{Z}_+$, there exists an irreducible nine-dimensional family $\mathcal{F}_m$ of K3 surfaces such that, for every $X_m\in\mathcal{F}_m$, \begin{center}
    $\ns(X_m)\cong U\oplus E_8(-2)\oplus <-4(m+1)>$. 
\end{center}Moreover, $X_m$ covers an Enriques surface of base change type $Y_m$ admitting a rational curve $B_{Y,m}$ of arithmetic genus $m$ which splits in $X_m$ in two smooth $(-2)$-curves $s_1$ and $s_2$ with $s_1\cdot s_2=2m$. Moreover, there exists an elliptic pencil $|F|$ on $Y_m$ such that $B_{Y,m}\cdot F=2$.
\end{theorem}
\begin{proof}
    See \cite[Section 3]{HS}.
\end{proof}
We shall say that $Y_m$ is an Enriques surface of base change type (or an $m$-special Enriques surfaces), that $B_{Y,m}$ is an $m$-special curve for $Y_m$ and that $|F|$ is an $m$-special elliptic pencil on $Y_m$. The family $\mathcal{F}_0$ is precisely the nine-dimensional family of K3 surfaces covering the Enriques surfaces which admit smooth $(-2)$-curves (called \textit{special} or \textit{nodal}). \par
In \cite{Pesa} the author gives a purely geometrical lattice free proof of the existence of the Enriques surfaces of base change type, that allows to prove that $B_{Y,m}$ is nodal (meaning that it has just simple nodes as singularities) and to show the existence of many other nodal rational curves of arbitrary large arithmetic genera, being bisections for some elliptic pencil, on every $m$-special Enriques surface $Y_m$.

\subsection{\rm L\sc inear systems on \rm E\sc nriques surfaces}

We give the definition of $\phi$-invariant for a nef line bundle of an Enriques surface, introduced by Cossec in \cite{Co}.
\begin{definition}
    Let $H\in\pic(Y)$ be a nef divisor of $Y$. Then, the \textit{$\phi$-invariant of H} is defined to be \begin{center}
        $\phi(H):=\min\{E\cdot H|E^2=0,E>0\}\in\mathbb{Z}$.
    \end{center}
\end{definition}
Here $E$ may be taken to be a half-fiber of a genus 1 pencil of $Y$, whence $2\phi$ is the minimum of the intersection between $H$ and a genus 1 pencil.

\begin{remark}\label{phi1}
   Let $Y_m$ be an Enriques surface of base change type and $B_{Y,m}\subset Y_m$ be an $m$-special curve. Since $B_{Y,m}\cdot F=2$ for an elliptic pencil $|F|$ on $Y_m$, we have that the $\phi(B_{Y,m})=1$. 
\end{remark}

The next result due to Galati and Knutsen (see \cite[Theorem 1.1]{GK}) states a necessary condition for the existence of a rational curve in the very general Enriques surface. Here very general just means that there exists a set that is the complement of a countable union of proper Zariski-closed subsets in the moduli space of Enriques surfaces satisfying the given conditions. 
\begin{theorem}\label{galknu}(Galati, Knutsen)
    Let $Y$ be a very general Enriques surface. If $C_Y\subset Y$ is an irreducible rational curve, then $C_Y$ is 2-divisible in $\numer(Y)$.
\end{theorem}
As a consequence of this Theorem, if $C$ is a rational curve on the very general Enriques surface, then $\phi(C)\neq 1$.

\begin{remark}
If $C_Y$ is a rational curve on a very general Enriques surface $Y$, then its arithmetic genus $o_a(C_Y)$ is such that \begin{center} $p_a(C_Y)\equiv_4 1$. \end{center} 
In fact, by Theorem \ref{galknu}, $C_Y$ is 2-divisible in $\numer(Y)$, i.e., there exists an effective divisor $D_Y$ on $Y$ such that $C_Y\sim 2D_Y$. We have \begin{center}
    $p_a(C_Y)=\frac{1}{2}C_Y^2+1=\frac{1}{2}(4D_Y^2)+1=2D_Y^2+1$.
\end{center}
Since for an Enriques surface $D_Y^2$ is even for every effective divisor $D_Y$, we have the assertion.

\end{remark}

The following is given in \cite[Proposition 2.6.1]{CD}. Recall that an Enriques surface $Y$ is said to be nonnodal if it does not contain $(-2)$-curves.
\begin{lemma}(Cossec, Dolgachev)\label{isotr}
    Assume $L$ is a big and nef divisor on a nonnodal Enriques surface $Y$ such that $L^2=2(p-1)>1$, where $p$ denotes the arithmetic genus of $L$. If $\phi(L)=1$, then $L$ is such that:
    \begin{center}
        $L\sim(p-1)E_1+E_2$,
    \end{center}
    where $E_1$ and $E_2$ are half-fibers on $Y$ with $E_1\cdot E_2=1$.
\end{lemma}
The previous result is also proved in \cite[Section 4]{CDGK3}, where the authors characterize any line bundle on an Enriques surface $Y$ in terms of their decompositions in $\pic(Y)$ in sum of particular divisors (that have to be half-fibers if $Y$ is nonnodal).
\subsection{\rm R\sc epresentation as double covering of quadrics}\label{sub 2.2}

We review and comment the model of the K3 surfaces covering an Enriques surface as a double cover of $\mathbb{P}^1\times\mathbb{P}^1$ presented by Barth, Peters, Hulek and Van de Ven in \cite{BHPV}.\par
Let $Q\cong\mathbb{P}^1\times\mathbb{P}^1$ be a smooth quadric in $\mathbb{P}^3$ and let $((x_o:x_1),(y_0:y_1))$ be bihomogeneous coordinates on it. We define the involution $\sigma$ on $Q$ by \begin{center}
   $\sigma((x_0:x_1),(y_0:y_1))=((x_0:-x_1),(y_0:-y_1))$.
\end{center}

\begin{theorem}\label{modelbarth}(Horikawa's representation of nonnodal Enriques surfaces)
    Let $Y$ be a nonnodal Enriques surface, $f:X\rightarrow Y$ be the universal covering, $\tau$ be the Enriques involution and $E_1,E_2\subset Y$ be two half-fibers with $E_1\cdot E_2=1$.  
If $\tau$ and $\sigma$ are defined as above, then there is a $\sigma$-invariant bihomogeneous polynomial of bidegree $(4,4)$ in $(x_o,x_1),(y_0,y_1)$ with zero-set $B$ on the smooth quadric $Q$, such that the universal covering $X$ of $Y$ is the minimal resolution of the double covering of $Q$ ramified over $B$. The curve $B$ is reduced with at worst simple singularities and does not contain any fixed point of $\tau$. The involution $\tau$ on $X$ is induced by the involution $\sigma$ on $Q$. The two rulings of $Q$ define the two genus 1 pencils $f^*E_1$ and $f^*E_2$ on $X$.
\end{theorem}
\begin{proof}
    See \cite[Theorem 18.1]{BHPV}.
\end{proof}
We denote by $\pi:X\rightarrow Q$ the described double covering. Theorem \ref{modelbarth} has a converse: as shown in \cite[V, Section 23]{BHPV}, given a $\sigma$-invariant curve as in the Theorem, the K3 surface $X$ and an Enriques surface $Y=X/\tau$ can be constructed from it.
We want now to apply this model to the case of Enriques surfaces of base change type.  Sometimes, we shall write $B_{E_1,E_2}$ to indicate the branch locus of the double cover $\pi:X\rightarrow Q$ associated to the two half-fibers $E_1$ and $E_2$. 

\begin{remark}
    If $Y_m$ is an $m$-special Enriques surface and $E_1$ is a half-fiber of an $m$-special genus 1 pencil of $Y_m$, then there exists a half-fiber $E_2\subset Y_m$ such that $E_1\cdot E_2=1$. This holds for every half-fiber on an Enriques surface, see for example \cite[Section 3]{CDGK3}.
\end{remark}

 All the nine-dimensional families $\mathcal{F}_m$ are irreducible and in particular the general K3 surface $X_m\in\mathcal{F}_m$ does not belong to $\mathcal{F}_0$, so that the Enriques surface $Y_m$ is nonnodal.

\begin{lemma}\label{bsmooth}
    If $Y$ is nonnodal, then $B_{E_1,E_2}$ is smooth for every choice of $E_1$ and $E_2$. In particular, the assertion holds for a general $m$-special Enriques surface $Y_m$. 
\end{lemma}

\begin{proof}
Suppose $B$ has a singular point $q$, that by Theorem \ref{modelbarth} turns out to be a simple singularity. Let $L_1\in|\mathcal{O}_Q(1,0)|$ and $L_2\in|\mathcal{O}_Q(0,1)|$ be the two members of the rulings intersecting $B$ at $q$. Let us denote by $\pi':X'\rightarrow Q$ the double covering of $Q$ ramified over $B$. Since the branch locus $B$ is singular at $q$, also $X'$ has a singularity at $q$. The K3 surface $X$, that is the universal covering of $Y$, is the minimal resolution of $X'$: let us denote by $E_q\subset X$ the exceptional locus over $q$ and by $\tilde{L}_1\subset X$ and $\tilde{L}_2\subset X$ the strict transforms of $(\pi')^{-1}(L_1)\subset X'$ and $(\pi')^{-1}(L_2)\subset X'$ respectively. Since $q$ is a simple singularity, $E_q$ consists of a chain of $(-2)$-curves. As usual $f$ denote the Enriques quotient: the genus 1 pencil $f^*(\epsilon_1)$ on $X$, that is the pullback of the ruling $|\mathcal{O}(1,0)|$, has a reducible member consisting of the union of $E_q$ and $\tilde{L}_1$, while the genus 1 pencil $f^*(\epsilon_2)$ on $X$ has $E_q+\tilde{L}_2$ as singular fiber. This proves that both the pencils $f^*(\epsilon_1)$ and $f^*(\epsilon_2)$ have a member consisting of a configuration of at least two $(-2)$-curves. Since they share the exceptional locus $E_q$, these singular members cannot be invariant with respect to $\tau$. So $\tau$ sends, say, $E_q+\tilde{L}_1$ to another member of $f^*(\epsilon_1)$ and hence the genus 1 pencil $\epsilon_1$ on $Y$ has a member consisting of at least two smooth rational curves. Since $Y$ is nonnodal, this is a contradiction.  
\end{proof}

\section{Rational curves on the very general Enriques surface\label{3.2}}

We introduce the so-called \textit{logarithmic Severi varieties}, parametrizing curves with given tangency conditions to a fixed curve. The definition and the main results are given by Dedieu in \cite{De} and they are based on the works of Caporaso and Harris (see for example \cite{CH}). 

\subsection{\rm L\sc ogarithmic \rm S\sc everi varieties of curves}

Let us denote by $\underline{N}$ the set of sequences $\alpha=[\alpha_1,\alpha_2,\dots]$ of nonnegative integers with all but finitely many $\alpha_i$ non-zero. In practice we shall omit the infinitely many zeroes at the end. For $\alpha\in\underline{N}$, we let \begin{center}
    $|\alpha|=\alpha_1+\alpha_2+\dots$\\
    $\mathcal{I}\alpha=\alpha_1+2\alpha_2+\dots+n\alpha_n+\dots$
\end{center}

\begin{definition}\label{logsev}
Let $S$ be a smooth projective surface, $T\subset S$ a smooth, irreducible curve
and $L$ a line bundle or a divisor class on $S$ with arithmetic genus $p$. Let $\gamma$ be an integer satisfying $0\leq\gamma\leq p$, let $\alpha\in\underline{N}$ such that \begin{center}
    $\mathcal{I}\alpha=L\cdot T$.
\end{center} We denote by $V_{\gamma,\alpha}(S,T,L)$ the locus of curves in $|L|$ such that \begin{itemize}
    \item $C$ is irreducible of geometric genus $\gamma$ and algebraically equivalent to $L$,
    \item denoting by $\mu:\tilde{C}\rightarrow S$ the normalization of $C$ composed with the inclusion $C\subset S$, there exists $|\alpha|$ points $Q_{i,j}\in C$, $1\leq j\leq\alpha_i$ such that \begin{center}
        $\mu^*T=\sum\limits_{1\leq j\leq\alpha_i}iQ_{i,j}$.
    \end{center}
\end{itemize}
\end{definition}

\begin{theorem}[Dedieu]\label{dedieu}
    Let $V$ be an irreducible component of $V_{\gamma,\alpha}(S,T,L)$, $[C]$ a general member of $V$ and $\mu:\tilde{C}\rightarrow S$ its normalization as in the Definition \ref{logsev}. Let now $Q_{i,j}$, $1\leq j\leq\alpha_i$ be points in $\tilde{C}$ such that \begin{center}$\mathcal{I}\alpha=L\cdot T$ \end{center} and set \begin{center}$D=\sum\limits_{1\leq j\leq\alpha_i} (i-1)Q_{i,j}$.\end{center}
    \begin{itemize}
        \item[(i)] If $-K_S\cdot C_i-$deg $\mu_*D_{|C_i|}\geq 1$ for every irreducible component $C_i$ of $C$, then \begin{center}
            $\dim(V)=-(K_S+T)\cdot L+\gamma-1+|\alpha|$
        \end{center}
        \item[(ii)] If $-K_S\cdot C_i-$deg $\mu_*D_{|C_i|}\geq 2$ for every irreducible component $C_i$ of $C$, then \begin{itemize}
            \item[(a)] the normalization map $\mu$ is an immersion, except possibly at the points $Q_{i,j}$;
            \item[(b)] the points $Q_{i,j}$ of $\tilde{C}$ are pairwise distinct;
            %\item[(c)] for every curve $G\subset S$, if $[C]$ is general with respect to $G$ then $C$ intersects $G$ transversely. 
        \end{itemize}

    \end{itemize}
\end{theorem}

\subsection{\rm R\sc ational curves}
Let $Y$ be a very general Enriques surface and $X$, $Q$ and $B$ be as in Subsection \ref{sub 2.2}. Let $f:X\rightarrow Y$ denote the Enriques quotient and $\pi:X\rightarrow Q$ the quotient over the smooth quadric. We investigate the problem of the existence of rational curves on $Y$. Let us suppose $Y$ has an irreducible rational curve $C_Y$.
\begin{lemma}\label{lemc'}
    If $C_Y\subset Y$ is an irreducible rational curve, then $f^{-1}(C_Y)$ consists of two linearly equivalent rational curves $C_X$ and $C'_X$ on $X$.
\end{lemma}
\begin{proof}
We follow the idea in \cite[Section 2]{CDGK}. We denote by $\nu_{C_Y}:\overline{C}_{y}\rightarrow C_Y$ the normalization of $C_Y$ and define $\eta_{C_Y}:=\mathcal{O}_{C_Y}(K_S)=\mathcal{O}_{C_Y}(-K_S)$, a nontrivial 2-torsion element in $\pic^0(C_Y)$, and $\eta_{\overline{C}_{Y}}:=\nu^*\eta_{C_Y}$. By standard results on covering of complex manifolds (see \cite[Section 2]{CDGK} or \cite[Section 1.17]{BHPV}), two cases may happen: \begin{itemize}
    \item $\eta_{\overline{C}_Y}\ncong\mathcal{O}_{\overline{C}_Y}$ and $f^{-1}(C_Y)$ is irreducible;
    \item $\eta_{\overline{C}_Y}\cong\mathcal{O}_{\overline{C}_Y}$ and $f^{-1}(C_Y)$ consists of two irreducible components conjugated by the Enriques involution $\tau$. 
\end{itemize} 
Since $C_Y$ is rational, its normalization $\overline{C}_Y$ is isomorphic to $\mathbb{P}^1$, whence any degree 0 line bundle on it is linearly equivalent to $\mathcal{O}_{\overline{C}_Y}$: it implies that $f^{-1}(C)$ consists of two rational curves $C_X$ and $C'_X$. Moreover, since $Y$ is very general, we have that $\tau^*$ acts as the identity on $\ns(X)$, so that $C_X\sim C'_X$.
\end{proof}
\begin{remark}\label{no-2}
As an immediate Corollary, we recover the very well-known fact that a very general K3 surface $X$ with an Enriques involution does not admit smooth $(-2)$-curves: otherwise, its image under the Enriques quotient would be a rational curve splitting in two $(-2)$ curves, which are in fact not linearly equivalent.
\end{remark}

The images $C_Q:=\pi(C_X)$ and $C'_Q:=\pi(C'_X)$ on $Q\cong\mathbb{P}^1\times\mathbb{P}^1$ under the Horikawa model $\pi:X\rightarrow Q$ have to be rational curves. In order to study the curves $C_Y$ on $Y$ and $C_X$ and $C'_X$ on $X$, we begin our analysis by starting from the curves $C_Q$ and $C'_Q$ on $Q$. The curves $C_Q$ and $C'_Q$ are linearly equivalent, so that both of them belong to $|\mathcal{O}_Q(m,n)|$, for some nonnegative integers $m$ and $n$.

\begin{lemma}\label{nosplit}
    Suppose that $m=1$ or $n=1$. Then, $\pi^{-1}(C_Q)=C_X$, with $C_X$ irreducible and rational on $X$. Equivalently, there exists an intersection point between  $C_Q$ and the branch locus $B$ of $\pi:X\rightarrow Q$ at which the intersection has odd order.
\end{lemma}
\begin{proof}
    The first statement is equivalent to the second since a curve in $Q$ splits in the K3 cover $X$ if and only if it intersects the branch locus $B$ with even order at every intersection point. 
    Suppose that $\pi^{-1}(C_Q)$ consists of two curves $C_X$ and $\tilde{C}_X$. Now, if $m=1$ or $n=1$, we have that the curve $C_Q$ is a section for one of the rulings $\mathcal{O}_Q(0,1)$ or $\mathcal{O}_Q(1,0)$. Hence, we have, say, $C_Q\cdot\mathcal{O}_Q(1,0)=1$, from which $(C_X+\tilde{C}_X)\cdot\pi^*\mathcal{O}_Q(0,1)=2$. Since $C_X\cdot\pi^*\mathcal{O}_Q(0,1)=0$ would imply that $C_X$ is a member of the elliptic pencil $|\pi^*\mathcal{O}_Q(0,1)|$ and the same for $\tilde{C}_X$, we have $C_X\cdot\mathcal{O}_Q(0,1)=\tilde{C}_X\cdot\mathcal{O}_Q(0,1)=1$, so that the rational curves $C_X$ and $\tilde{C}_X$ are sections for the elliptic pencil $|\pi^*\mathcal{O}_Q(0,1)|$ and whence smooth $(-2)$-curves. Since $X$ covers a very general Enriques surface $Y$, this is a contradiction.
\end{proof}

\begin{proposition}\label{c''}
Suppose that $m=1$ or $n=1$ and denote by $p$ an intersection point  between $C_Q$ and $B$ at which the intersection has odd order $k$. Then, there exists another intersection point $p'$ between $C_Q$ and $B$ at which the intersection has odd order $k'$. Furthermore, $C_Q\in|\mathcal{O}_Q(m,n)|$ is a rational curve tangent to $B$ at $r:=2m+2n-\frac{k-1}{2}-\frac{k'-1}{2}-1$ (possibly coincident) points $q_1,\dots,q_r$ with even order.
\end{proposition}
\begin{proof}
In order to show the existence of $p'$, it is sufficient to notice that the intersection number $C_Q\cdot B=4m+4n$ is even. We have \begin{center}
    $C_X^{2}=2C_Q^{2}=2(\mathcal{O}_Q(m,n)^2)=2(2mn)=2(2mn+1-1)$,
\end{center}from which the arithmetic genus of $C_X$ is \begin{center}
       $p_a(C_X)=2mn+1$.
   \end{center}
   Since $m$ or $n$ is equal to 1, the curve $C_Q$ is smooth as well as the branch locus $B$, and whence the singularities of $C_X$ come from points at which $C_Q$ is tangent to $B$. Suppose $C_Q$ is tangent to $B$ at a point $s$ with order $k_s$ and consider local coordinates $(x,y)$ for $Q$ in such a way $B$ has local equation $y=0$ at $s$: the curve $C_Q$ has local equation $y-x^{k_s}$ at $s$. If the double covering is given by $(x,t)\mapsto (x,t^2)$ in the local coordinates $(x,y,t)$, we have that $C_Q$ has local equation $t^2-x^{k_s}$ in a neighborhood of $\pi^{-1}(s)$. This implies that $C_X$ has an $A_{k_s-1}$ singularity at $\pi^{-1}(s)$. Now, if $k_s$ is even, an $A_{k_s-1}$ singularity imposes $\frac{k_s}{2}$ conditions on the genus of $C_X$, while if $k_s$ is odd an $A_{k_s-1}$ singularity imposes $\frac{k_s-1}{2}$ conditions. Since $C_X$ is rational, its geometric genus $g(C_X)$ is \begin{center}
      $g(C_X)=2mn+1-\frac{k-1}{2}-\frac{k'-1}{2}-\sum\limits_{j=1}^{r'} \frac{k_j}{2}-\sum\limits_{j=r'+1}^r \frac{k_j-1}{2}=0$,
  \end{center}where $k_j$ indicates the intersection order between $C_Q$ and $B$ at the point $q_j$, the first sum runs over the points at which the intersection has even order and the second sum runs over the points at which the intersection has odd order. We have \begin{center}
      $\sum\limits_{j=1}^{r'} \frac{k_j}{2}+\sum\limits_{j=r'+1}^r \frac{k_j-1}{2}=2mn+1-\frac{k-1}{2}-\frac{k'-1}{2}$=\end{center}\begin{center}$=2m+2n-1-\frac{k}{2}-\frac{k'}{2}+1=2m+2n-\frac{k}{2}-\frac{k'}{2}$,
  \end{center}where the second equality holds since for $m=1$ or $n=1$, we have $2mn+1=2m+2n-1$. But $C_Q\cdot B=\mathcal{O}(m,n)\cdot\mathcal{O}(4,4)=4m+4n$, so that we have \begin{center}
      $\sum\limits_{j=1}^r k_j=4m+4n-k-k'$,
  \end{center}from which 
  \begin{center}
      $\sum\limits_{j=1}^r \frac{k_j}{2}=2m+2n-\frac{k}{2}-\frac{k'}{2}$.
  \end{center}Hence, $r'=r$ and the sum $\sum\limits_{j=r'+1}^r \frac{k_j-1}{2}$ is empty, which means that the intersection order between $C_Q$ and $B$ at $q_j$ is even for every $q_j$.
\end{proof}

\begin{remark}
    In the situation of our interest, at least one between $m$ and $n$ will be equal to 1: in this case, all the irreducible curves of $|\mathcal{O}_Q(m,n)|$ are smooth, so they are isomorphic to their normalization. In other words, the condition of Proposition \ref{c''} is equivalent to requiring that $C_Q$ belongs to the logarithmic Severi variety $V_{0,\alpha}(\mathbb{P}^1\times\mathbb{P}^1,B,\mathcal{O}_Q(m,n))$, where $\alpha=[2,2mn+1]$. 
\end{remark}

By reversing the process, to prove the existence of rational curves on the very general Enriques surface $Y$, it is sufficient to find a curve in the linear system $|\mathcal{O}_Q(m,n)|$ satisfying the conditions of Proposition \ref{c''}. Notice that, as pointed out in the proof of Proposition \ref{c''}, if $m=1$ or $n=1$, we have that $2mn+1=2m+2n-1$.
 \par
The aim of the last part of the paper is to show that the logarithmic Severi variety $V_{0,\alpha}(\mathbb{P}^1\times\mathbb{P}^1,B,\mathcal{O}_Q(m,n))$, with $\alpha=[2,2m+2n-1]$, is nonempty for values of $m$ and $n$ (with at least one of them equal to 1) such that $mn$ is arbitrary large, or, equivalently, that the very general Enriques surface $Y$ admits rational curves of arbitrary large arithmetic genus. From now on, we shall write just $V_\alpha$ to indicate $V_{0,\alpha}(\mathbb{P}^1\times\mathbb{P}^1,B,\mathcal{O}_Q(m,n))$ (or $V_{\alpha,B}$ when we want to focus on the curve $B$).
\begin{lemma}\label{valpha}
    If $V_{0,\alpha}$ is nonempty, then its dimension is 0.\end{lemma} \begin{proof}
    Suppose that there exists $[C_Q]\in V_{0,\alpha}$. Since $m=1$ or $n=1$, we have that $C_Q$ is a smooth rational curve and whence it is isomorphic to its normalization. We have $\alpha_1=2$ and $\alpha_2=2m+2n-1$, so we let $Q_{1,1},Q_{1,2},Q_{2,1},\dots,Q_{2,2m+2n-1}$ be points in $C_Q$ such that, denoting by $\mu:C_Q\rightarrow Q$ the inclusion $C_Q\subset Q$, we have $\mu^*(B)=Q_{1,1}+Q_{1,2}+2Q_{2,1}+\dots+2Q_{2,2m+2n-1}$ as in Definition \ref{logsev}. We set $D=Q_{2,1}+\dots+Q_{2,2m+2n-1}$: by construction, $d:=\deg D_{|C''}=2m+2n-1$.
     We have that \begin{center}
        $-(K_{\mathbb{P}^1\times\mathbb{P}^1})\cdot C_Q-\deg\mu_*D_{|C_Q}=\mathcal{O}_Q(2,2)\cdot\mathcal{O}_Q(m,n)-(2n+2m-1)=2m+2n-2n-2m+1=1$,
    \end{center} where we set $D=Q_1+\dots+Q_{2n+2m-1}$. Then, by Theorem \ref{dedieu} \begin{center}
        $\dim(V_{0,\alpha})=-(K_{\mathbb{P}^1\times\mathbb{P}^1}+\mathcal{O}_Q(4,4))\cdot\mathcal{O}_Q(m,n)-1+|\alpha|=-2m-2n-1+2m+2n+1=0$.
    \end{center}
    \end{proof}

As announced in the introduction, our strategy will be to deform the $m$-special curves on the Enriques surfaces of base change type found by Hulek and Sch\"utt to rational curves on the very general Enriques surfaces, by exploiting the "regeneration" results about curves on K3 surfaces due to Chen, Gounelas and Liedtke.\par
Let $X_m\in\mathcal{F}_m$ be a general K3 surface covering an $m$-special Enriques surface $Y_m$. Denote by $B_{Y,m}\subset Y_m$ an $m$-special curve on $Y_m$ and let $f^{-1}(B_{Y,m})=s_1+s_2$. As stated in Theorem \ref{hs}, $s_1$ and $s_2$ are smooth $(-2)$ curves in $X_m$. Now, since $B_{Y,m}$ is a bisection for an elliptic pencil $F_1$ on $Y_m$ of arithmetic genus $m$, we have $\phi(B_{Y,m})=1$. Moreover, since $Y_m$ is nonnodal, by Lemma \ref{isotr} it is such that \begin{center}
    $B_{Y,m}\sim (m-1)E_1+E_2$,
\end{center}where $E_1$ and $E_2$ are two half-fibers such that $E_1\cdot E_2=1$ and $F_1\sim 2E_1$ is in fact the elliptic pencil of which $B_{Y,m}$ is a bisection. We consider the Horikawa representation of $X_m$ as double cover of a smooth quadric $Q$ described by Theorem \ref{modelbarth}, with the choice of $E_1$ and $E_2$ as half-fibers. $X_m$ admits the two 2:1 coverings \begin{center} $\pi:X_m\rightarrow Q$ and $f:X_m\rightarrow Y_m$. \end{center} We denote by $A_1:=f^*(E_1)$ and $A_2:=f^*(E_2)$ the elliptic pencils on $X_m$ given by the pullbacks of the half-fibers $E_1$ and $E_2$ respectively; moreover, we denote by $\tilde{A}_1\in|A_1|$ and $\tilde{A}_2\in|A_2|$ the preimages of the half-fibers $E_1$ and $E_2$ under $f$. With this notation, the pencils $A_1$ and $A_2$ are sent by $\pi$ to the rulings $\mathcal{O}_Q(1,0)$ and $\mathcal{O}_Q(0,1)$ respectively. 
    
\begin{proposition}\label{nonempty}
    For every $m\in\mathbb{Z}_+$, there exists a family of curves $\mathfrak{B}_m\subset|\mathcal{O}_Q(4,4)|$ such that for every $B_m\in\mathfrak{B}_m$, the logarithmic Severi variety $V_{0,\alpha}(\mathbb{P}^1\times\mathbb{P}^1,B_m,\mathcal{O}_Q(1,m-1))\neq\emptyset$, with $\alpha=[2,2m-1]$. 
\end{proposition}
\begin{proof}
Since the $m$-special curve $B_{Y,m}$ is such that $B_{Y,m}\sim (m-1)E_1+E_2$ on $Y_m$, we have $(s_1+s_2)\sim (m-1)\tilde{A}_1+\tilde{A}_2$. Moreover, we have $\pi_*(\tilde{A}_1)=2\mathcal{O}_Q(1,0)$ and $\pi_*(\tilde{A}_2)=2\mathcal{O}_Q(0,1)$, from which $\pi_*((m-1)\tilde{A}_1+\tilde{A}_2)=\mathcal{O}_Q(2m-2,2)$, so that also $\pi_*(s_1+s_2)=\mathcal{O}_Q(2m-2,2)$. But $s_1$ and $s_2$ are sections for the elliptic pencil $A_1$, whence they are sent to sections of the ruling given by $\mathcal{O}_Q(1,0)$. Hence, $\pi(s_1)\in|\mathcal{O}_Q(k,1)|$ and $\pi(s_1)\in|\mathcal{O}_Q(2m-2-k,1)|$. We want to show that $k=m-1$. We have \begin{center}
    $B_{Y,m}\cdot E_2=((m-1)E_1+E_2)\cdot E_2=m-1$,\end{center} from which \begin{center}
    $f^*(B_{Y,m})\cdot f^*(E_2)=(s_1+s_2)\cdot A_2=2m-2$.
\end{center}Furthermore, \begin{center}
    $s_1\cdot A_2=s_1\cdot f^*(E_2)=f_*(s_1)\cdot E_2=B_{Y,m}\cdot E_2=m-1$.
\end{center}But $s_1\cdot A_2$ is also such that \begin{center}
    $s_1\cdot A_2=s_1\cdot\pi^*(\mathcal{O}_Q(0,1))=\pi_*(s_1)\cdot\mathcal{O}_Q(0,1)=\pi(s_1)\cdot\mathcal{O}_Q(0,1)$,
\end{center}from which \begin{center}
    $\pi(s_1)\sim\mathcal{O}_Q(m-1,1)$.
\end{center}
The equality $\pi_*(s_1)\cdot\mathcal{O}_Q(0,1)=\pi(s_1)\cdot\mathcal{O}_Q(0,1)$ holds since $s_1$ is a $(-2)$-curve on $X_m$ and by Lemma \ref{bsmooth} the branch locus $B_{E_1,E_2}\in|\mathcal{O}_Q(4,4)|$ is smooth, so that the image under $\pi$ of $s_1$ is a rational curve intersecting $B_{E_1,E_2}$ with even order at every intersection point and thus it splits on $X_m$, or, equivalently, $\pi$ restricted to $s_1$ is birational.\par The family $\mathfrak{B}_m\subset|\mathcal{O}_Q(4,4)|$ consists of the branching curves corresponding to the described model for every $X_m\in\mathcal{F}_m$.
\end{proof}
\begin{remark}
    We point out that if $B_m\in\mathfrak{B}_m$, we have that $V_{\alpha,B_m}\neq\emptyset$ even for $\alpha=[0,2m]$, but for our purpose the nonemptiness of $V_{\alpha,B_m}$ for $\alpha=[2,2m-1]$ is sufficient. 
\end{remark}

To prove Theorem \ref{main}, we need the following definitions and results about stable maps on families of K3 surfaces. We follow \cite{Ko} for the original Definitions of stable maps and their moduli space and \cite{CGL} for the regeneration results about curves on families of K3 surfaces.\par

\begin{definition}
Let $V$ be a complex algebraic variety.\\
    A stable map is a structure $(C,x_1,\dots,x_k,f)$ consisting of a connected compact reduced curve $C$ with $k\geq 0$ pairwise distinct marked nonsingular points $x_i$ and at most ordinary double singular points, and a map $f:C\rightarrow V$ having no nontrivial first order infinitesimal automorphisms, identical on $V$ and $x_1,\dots,x_k$.
\end{definition}
The condition of stability means that every irreducible component of $C$ of genus 0
 (resp. 1) which maps to a point must have at least 3 (resp. 1) special (i.e. marked
 or singular) points on its normalization.
There is a notion of isomorphism of stable maps and of moduli space of stable maps due to Kontsevich: $\overline{M}_{p}(V,\beta)$ denotes the moduli stack of stable maps to $V$ of curves of arithmetic genus $p\geq 0$ with $k\geq 0$ marked points such that $f_*[C]=\beta$.\par
 The next result, due to Chen, Gounelas and Liedtke, is our main tool to prove Theorem \ref{main}.

 \begin{theorem}\label{cgl}
     Let $\mathcal{F}\rightarrow\mathcal{V}$ be a smooth projective family of K3 surfaces over an irreducible base scheme $\mathcal{V}$ and let $\mathcal{H}$ be a line bundle on $\mathcal{F}$. Then, every irreducible component $M\subset \overline{M}_{p}(\mathcal{F}/\mathcal{V},\mathcal{H})$ satisfies \begin{center}
         $\dim M\geq p+\dim\mathcal{V}$.
     \end{center}
 \end{theorem}
 \begin{proof}
     See \cite[Theorem 2.11]{CGL}.
 \end{proof}
 We are able to prove Theorem \ref{main}.
\begin{proof} [Proof of Theorem \ref{main}]

 We consider a one-dimensional family $\mathcal{Q}$ over a curve $\mathcal{C}$ such that for every $t\in\mathcal{C}$, \begin{center}
    $\mathcal{Q}_t\cong Q\cong\mathbb{P}^1\times\mathbb{P}^1$.
\end{center}Moreover, we consider the line bundle over $\mathcal{Q}$ \begin{center}
    $\mathcal{R}:=(\mathcal{Q},\mathcal{O}_{\mathcal{Q}}(4,4))$
\end{center}defined in the following way: for every $t\in\mathcal{C}$, the restriction of $\mathcal{R}$ to $\mathcal{Q}_t$ is \begin{center}
    $\mathcal{R}_t\cong(\mathcal{Q}_t,\mathcal{O}_{\mathcal{Q}_t}(4,4))$.
\end{center}
For every $m\in\mathbb{Z}_+$, we consider another line bundle $\mathcal{H}^m_{\mathcal{Q}}$ over $\mathcal{Q}$ defined in the following way: for every $t\in C$, the restriction of $\mathcal{H}^m_{\mathcal{Q}}$ to $\mathcal{Q}_t$ is \begin{center}
    $\mathcal{H}^m_{{\mathcal{Q}},t}\cong(\mathcal{Q}_t,\mathcal{O}_{\mathcal{Q}_t}(1,m-1))$.
\end{center}
Moreover, we fix a point $t_0\in\mathcal{C}$ and consider a curve $B_m\in\mathcal{R}_{t_0}=(\mathcal{Q}_{t_0},\mathcal{O}_{\mathcal{Q}_{t_0}}(4,4))$ such that $V_{0,\alpha}(\mathcal{Q}_{t_0},B_m,\mathcal{O}_{\mathcal{Q}_{t_0}}(1,m-1))\neq\emptyset$ (then a curve $B_m$ belonging to the family $\mathfrak{B}_m\subset|\mathcal{O}_{\mathcal{Q}_{t_0}}(4,4)|$ as in Proposition \ref{nonempty}). Finally, we let $\mathcal{W}$ with

\begin{center} \begin{tikzcd} 
  \mathcal{W} \arrow{r} \arrow{dr}  & \mathcal{R} \arrow{d}\\
& \mathcal{C}
  \end{tikzcd} \end{center}
 be a section of the relative line bundle $\mathcal{R}$, such that $\mathcal{W}_{t_0}=B_m$.

Let now $\mathcal{X}\rightarrow\mathcal{C}$ be the family over $\mathcal{C}$ of K3 surfaces such that for every $t\in\mathcal{C}$,
$\mathcal{X}_t$ is the K3 surface obtained as the double cover of $\mathcal{Q}_t$ ramified over $\mathcal{S}_t$ and let $\pi$ 

\begin{center} \begin{tikzcd} 
  \mathcal{X} \arrow{r}{\pi} \arrow{dr}  & \mathcal{Q} \arrow{d}\\
& \mathcal{C}
  \end{tikzcd} \end{center}
be the relative double covering. 
For every $m\in\mathbb{Z}_+$, we define the line bundle on $\mathcal{X}$ \begin{center}
$\mathcal{H}^m_\mathcal{X}:=\pi^*\mathcal{H}^m_\mathcal{Q}$
\end{center}to be the pullback of $\mathcal{H}^m_\mathcal{Q}$ via $\pi$. In other words, the restriction of $\mathcal{H}^m_{\mathcal{X}}$ to every $\mathcal{X}_t$ is the pullback of $(\mathcal{Q}_t,\mathcal{O}_{\mathcal{Q}_t}(1,m-1))$. \par Since $V_{0,\alpha}(\mathcal{Q}_{t_0},B_m,\mathcal{O}_{\mathcal{Q}_{t_0}}(1,m-1))\neq\emptyset$, the line bundle $\mathcal{H}^m_{{\mathcal{X}},t_0}$ admits a special member consisting of two smooth rational curves, that we call $s_1$ and $r_1$. \\ %We want to show that this curve deforms in the family $\mathcal{X}\rightarrow C$ as a rational curve by using Theorem \ref{cgl}.\\ 
Let now $L$ be the union of two copies of $\mathbb{P}^1$ meeting at one point and consider the map $f:L\rightarrow \mathcal{X}_{t_0}$ such that $f(L)=s_1+r_1$. Since there are no contracted components, $(f,L)$ is a stable map.
Now, consider the moduli stack \begin{center}
    $\overline{M}_{o}(\mathcal{X}/\mathcal{C},\mathcal{X}_{\mathcal{H}})$.
\end{center} We proved that it is nonempty, whence by Theorem \ref{cgl} every irreducible component of it has dimension $\dim\mathcal{C}$, that is 1. By Lemma \ref{valpha}, the fibers over the points of $\mathcal{C}$ have dimension $0$, thus, for every $t$, the K3 surface $\mathcal{X}_t$ is such that the line bundle $\mathcal{X}_{\mathcal{H},t}$ has a rational member, that is sent to a rational curve on $\mathcal{Q}_t$. \par We actually proved that for every $t\in\mathcal{C}$, the curve $\mathcal{W}_t\in|\mathcal{O}_{\mathcal{Q}_t}(4,4)|$ is such that $V_{0,\alpha}(\mathcal{Q}_t,\mathcal{W}_t,\mathcal{O}_{\mathcal{Q}_t}(1,m-1))\neq\emptyset$. Since $\mathcal{W}$ is an arbitrary section of the relative line bundle $\mathcal{R}$ over $\mathcal{Q}$, the nonemptiness of $V_{0,\alpha}(Q,B,\mathcal{O}_Q(1,m-1))$ holds in particular for every $B\in|\mathcal{O}_Q(4,4)|$ such that the double cover of $Q$ ramified over $B$ is a K3 surface with an Enriques involution. Let then $B$ be such a curve and let $C_{Q,m}\in V_{0,\alpha}(Q,B,\mathcal{O}_Q(1,m-1))$. As in the proof of Lemma \ref{nosplit}, the very general K3 surface covering an Enriques surface does not admit smooth $(-2)$-curves, whence the smooth rational curve $C_{Q,m}$ does not split and in particular it is not totally tangent to $B$. This implies that the preimage $C_{X,m}:=\pi^{-1}(C_{Q,m})$ is irreducible. Moreover, by Lemma \ref{lemc'}, the Enriques involution identifies $C_{X,m}$ with another rational curve $C'_{X,m}$ linearly equivalent to $C_{X,m}$; let us denote by $C'_{Q,m}$ the image of $C'_{X,m}$ under $\pi$ and by $C_{Y,m}$ the image of $C_{X,m}$ (and of $C'_{X,m})$ under $f$. We have that the arithmetic genus of $C_{Y,m}$ is
\begin{center}
    $p_a(C_{Y,m})=\frac{1}{2}(C_{Y,m})^2+1=\frac{1}{4}(C_{X,m}+C'_{X,m})^2+1$,
\end{center}
and 
\begin{center}
    $(C_{X,m}+C'_{X,m})^2=2(C_{Q,m}+ C'_{Q,m})^2=2((\mathcal{O}_Q(1,m-1))^2+(\mathcal{O}_Q(1,m-1))^2+2(\mathcal{O}_Q(1,m-1))^2)=16m-16$,
\end{center}so that, finally,
\begin{center}
    $p_a(C_{Y,m})=\frac{1}{4}(16m-16)+1=4m-3$.
\end{center} The $\phi$-invariant of $C_{Y,m}$ is 2 and it is computed by the half-fiber $E_1$ on $Y$ induced by the ruling $\mathcal{O}(0,1)$ of $Q$. Indeed, we have \begin{center} $C_{Y,m}\cdot E_1=\frac{1}{2}(C_{X,m}+C'_{X,m})\cdot\pi^*E_1=(C_{Q,m}+C'_{Q,m})\cdot\mathcal{O}(0,1)=$\\
$=\mathcal{O}(2,2m-2)\cdot\mathcal{O}(0,1)=2$.\end{center}
\end{proof}
%As we mentioned at the beginning of the Section, the proof of Theorem \ref{main} holds for every K3 surface that is a double cover of a smooth quadric $Q\subset\mathbb{P}^3$ ramified along a curve of bidegree $(4,4)$.

  \end{document}